\documentclass[11pt]{amsart}

\usepackage{amsmath, a4wide}
\usepackage{amssymb,amsthm,amsfonts}
\usepackage{amsrefs}
\usepackage[utf8]{inputenc}
\usepackage[T1]{fontenc}
\usepackage{xcolor}

\renewcommand{\mathcal}{\mathscr}

\def\R {\mathbb{R}}

\def\Z {\mathbb{Z}}

\renewcommand{\epsilon}{\varepsilon}
\newcommand{\eps}{\varepsilon}
\newcommand{\e}{\varepsilon}

\renewcommand{\leq}{\leqslant}
\renewcommand{\le}{\leqslant}
\renewcommand{\geq}{\geqslant}
\renewcommand{\ge}{\geqslant}

\newtheorem{proposition}{Proposition}[section]
\newtheorem{theorem}[proposition]{Theorem}
\newtheorem{corollary}[proposition]{Corollary}
\newtheorem{lemma}[proposition]{Lemma}
\theoremstyle{definition}
\newtheorem{definition}[proposition]{Definition}

\newtheorem{remark}[proposition]{Remark}
\numberwithin{equation}{section}

\title{Fractional mean curvature flow of Lipschitz graphs} 

\author[A. Cesaroni, M. Novaga]{}


 \email{annalisa.cesaroni@unipd.it}
 \email{matteo.novaga@unipi.it}

\thanks{The authors were supported by the INDAM-GNAMPA and by the PRIN Project 2019/24 {\it Variational methods for stationary and evolution problems with singularities and interfaces}.}

\begin{document}
\maketitle

\centerline{\scshape Annalisa Cesaroni }
\medskip
{\footnotesize
 \centerline{Department of Statistical Sciences}
   \centerline{University of Padova}
   \centerline{Via Cesare Battisti 141, 35121 Padova, Italy  }
} 
\medskip

\centerline{\scshape Matteo Novaga}
{\footnotesize
\centerline{ Department of Mathematics}
   \centerline{University of Pisa}
   \centerline{Largo Bruno Pontecorvo 5, 56127 Pisa, Italy  }
}
%
%
\begin{abstract}
We consider the fractional mean curvature flow of entire Lipschitz graphs. 
We provide regularity results, and we study the long time asymptotics of the flow. In particular we show that in a suitable  rescaled framework, if the initial graph is a sublinear perturbation of a cone,   the evolution asymptotically approaches   an expanding self-similar solution. We also prove stability of hyperplanes and of convex cones in the unrescaled setting. 
\end{abstract} 

\tableofcontents

\section{Introduction}
Given a set $E\subseteq \R^{n+1}$, we define the fractional mean curvature flow $E_t$ starting from  $E$ as the flow obtained by the following geometric evolution law: 
the velocity at a point~$p\in \partial E_t$ is given by
\begin{equation}\label{kflow} 
\partial_t p\cdot \nu(p)=-H_s(p,E_t):=-\lim_{\e\to 0}
\int_{\R^{n+1}\setminus B_\eps(p)}\Big(
\chi_{\R^{n+1}\setminus E_t}(y)-\chi_{E_t}(y)\Big)\ \frac{1}{|p-y|^{n+1+s}}\,dy,\end{equation} 
where $s\in (0,1)$ is a fixed parameter and $\nu(p)$ is the outer normal to $\partial E_t$ at $p$.  

The fractional mean curvature flow  can be interpreted as the fractional analogue of the classical mean curvature flow. Indeed, as the mean curvature flow is the $L^2$ gradient flow of the perimeter, the fractional mean curvature flow is the $L^2$ gradient  flow of   the so-called fractional perimeter, see \cite{MR2675483}, which can be seen as an interpolation  norm (the Gagliardo fractional seminorm) of the characteristic function of a measurable set, which interpolates between the BV norm, which is the standard perimeter, and the $L^1$ norm, which is the volume. 

Therefore, the fractional mean curvature flow presents  some analogies with the classical mean curvature flow. Recently a local existence result for smooth solutions starting from compact $C^{1,1}$ initial sets was provided in  \cite{lamanna} (see also \cites{cs, cnr} for an analog of the BMO scheme),
moreover existence and uniqueness of the level set flow for general nonlocal evolution equations,  including \eqref{kflow}, has been developed in   \cite{cmp} (see also \cite{i}), using the maximum principle and the monotonicity of the curvature with respect to inclusions. 
On the other hand, the  fractional flow  presents some different features with respect to the classical mean curvature flow, since  nonlocal effects  
come into play. For instance, as a basic example one can consider a  planar strip, which is stationary for the curvature flow, and shrinking for the fractional flow.  

An important issue in the study of the fractional flow, as for the classical one, is the investigation of the long time   behavior of solutions and  the analysis of the formation of singularities.  In the local case, one of the most important tool in this analysis is  the monotonicity formula established by Huisken in \cite{hu}. The analog of such formula in the fractional setting is still an open problem.  As a consequence, it is still missing a systematic approach for the study of long time asymptotics of the flow and the classification of the possible singularities which may appear.  Nevertheless some results have been recently obtained, we recall for instance  the analysis of the formation of neckpinch singularities in \cite{csv1}, and of the fattening phenomenon for the  evolution of curves  with cross-type and cusp-type singularities in \cite{cdnv}.  Moreover, in \cite{csv2}   it has been proved that smooth convex sets  evolving under the  volume preserving fractional  mean curvature flow approach round spheres. 

Here we analyze the evolution \eqref{kflow} under the additional assumption that the boundary of the initial  datum $E_0$ can be written as a Lipschitz graph, that is, there exists  $e\in \R^{n+1}$, such that $\nu(p)\cdot e>0$  for every $p\in \partial E_0$. By monotonicity of the flow it is possible  to show that the evolution $E_t$ maintains  this property for all positive times
$t>0$, that is,   $\nu(p)\cdot e>0$  for every $p\in \partial E_t$, see Section \ref{sectionlevel}. Up to a rotation of coordinates, we will assume that $e=e_{n+1}$.

For the local case, the analysis of the mean curvature flow of entire Lipschitz graphs   goes back to the work by Ecker and Huisken \cites{eh,e}, whereas the anisotropic mean curvature flow of entire Lipschitz graphs has been recently considered by the authors and a coauthor in \cite{ckn}.  In particular, in \cite{eh}  it is proved that   the evolution admits  a smooth 
 solution for all times, which approaches a self-similar solution as $t\to +\infty$, provided that the
initial graph is  "straight" at infinity, in the sense that is a sublinear perturbation of a cone, see assumption \eqref{straight} below. 
In this paper we provide analogous results in the fractional setting. In particular, in Section \ref{sectionreg} we prove the
 $C^{1,\alpha}$ regularity of the flow starting from a Lipschitz graph, which can be strengthened to $C^\infty$  if the initial graph enjoys more regularity. These results are based on the fact that the fractional curvature is an elliptic operator, and so we may apply the regularity results for nonlinear fractional parabolic problems obtained in \cite{serra, silv} and the parabolic bootstrap argument developed in \cite{lamanna}.   
In the case of initial graphs which are merely Lipschitz continuous,  we do not recover the $C^\infty$ regularity obtained in the local case, except in the case of
 self-similar solutions (see Theorem \ref{omogenee}), and this is due to the fact that a  parabolic bootstrap regularity  argument  is  missing  in this setting  for quasilinear fractional operators.  
  
Finally, in Section \ref{sectionconv} we provide the convergence of the rescaled solution to a self-similar expanding solution, under the assumption that the initial graph is straight at infinity in the sense of Ecker and Huisken.  We recall here what we mean for self-similar expanding or contracting solution to \ref{kflow}. 

\begin{definition} \label{defhom} 
An  expanding homothetic solution  is a solution    to \eqref{kflow} such that   $E_t= \lambda(t)E_1$ where $\lambda(1)=1$ and  $\lambda'(t)\geq 0$ for $t>1$. 
This is equivalent to assume that  $E_1$ is a solution to \begin{equation}\label{exp} c(  p\cdot \nu)=-  H_s(p  , E_1)\end{equation}  for some $c\geq 0$. Observe that  necessarily  $\lambda(t)=\left[ c(s+1) (t-1)+1\right]^{\frac{1}{s+1}}$. 

A shrinking homothetic solution   to \eqref{kflow}  is a solution to \eqref{kflow} such that   $E_t=  \lambda(t)E_1$ where $\lambda(1)=1$ and  $\lambda'(t)\leq 0$ for $t>1$. 
This is equivalent to assume that  $E_1$ is a solution to \begin{equation}\label{contr} c(  p\cdot \nu)= H_s(p  , E_1)\end{equation}  for some $c\geq 0$. 
  \end{definition} 
  
In Section \ref{sectionhom} we study the main properties of the expanding self-similar solutions to \eqref{kflow}, whereas in Section \ref{sectionsh} we show that the only graphical shrinking self-similar solutions to \eqref{kflow} are actually stationary solutions. In the local setting, this result has been obtained  for entire graphs without growth condition at infinity in \cite{w}.  In the fractional setting we obtain the result for entire Lipschitz graphs as a byproduct of a Liouville theorem for ancient solutions of parabolic nonlinear equations obtained in \cite{serra}.   We also recall that a preliminary analysis of existence and stability of fractional symmetric shrinkers has been developed  in \cite{cn}. 

\smallskip

Finally  Section \ref{unrescaled} contains convergence  results  in the unrescaled setting. In particular we provide the stability of hyperplanes, when we start the evolution from graphs which are asymptotically flat. In the local setting this stability can be proved by using comparison with large balls and area decay estimates, see \cites{clutt, e, n}, whereas in the fractional setting  we use comparison with large balls and  an argument based on construction of appropriate periodic barriers. 
We show also   stability of  convex cones and, in some particular cases, of mean convex cones,  in the unrescaled setting. Analogous results in the  local   setting were  obtained  in \cite{clutt} for the isotropic case and in \cite{ckn} for the anisotropic case.   

\section{Level set formulation} \label{sectionlevel} 
The level set flow associated to~\eqref{kflow} can be defined as follows. 
Given an initial set $E_0\subset \R^{n+1}$ we choose a bounded Lipschitz continuous function 
$U_{0}:\R^{n+1}\to \R$ such that 
\begin{eqnarray*}&&
\partial E_0=\{p\in\R^{n+1} \text{ s.t. } U_{0}(p)=0\}=\partial\{p\in\R^{n+1}\text{ s.t. }  U_0(p)\geq 0\}\\
{\mbox{and }} && E_0=\{p\in\R^{n+1}\text{ s.t. } U_0(p)\geq 0\}. 
\end{eqnarray*}
Let also~$U(p,t)$ be the viscosity solution of the following nonlocal parabolic problem
\begin{equation}\label{level1}
\begin{cases}
U_t(p,t)+|DU(p,t)| H_s((p,t),  \{ (p',t)\ | U(p',t)\geq U(p,t)\})=0\\
U(p,0)= U_0(p).
\end{cases} 
\end{equation} 
Then the level set flow of $\partial E_0$ is given by 
\begin{equation}\label{sigmaet} \Sigma_E(t):=\{p\in\R^{n+1}\text{ s.t. }  U(p,t)=0\}. \end{equation} 
We associate to this level set the outer and inner flows defined as follows:
\begin{equation}
\label{outin} E^+(t):= \{p\in\R^{n+1}\text{ s.t. }  U(p,t)\geq 0\} \quad
{\mbox{and}} \quad E^-(t):= \{p\in\R^{n+1}\text{ s.t. }  U(p,t)>0\}.
\end{equation} 
We observe that the equation in \eqref{level1} is geometric, so  
if we replace the initial condition with any function $V_0$ with the same level sets $\{U_0 \geq 0\}$ and $\{ U_0 > 0 \}$, the evolutions
$E^+(t)$ and $E^-(t)$ remain the same.
Existence and uniqueness of viscosity solutions to the level set formulation of \eqref{kflow} has been provided in \cites{i,cmp}, and qualitative properties of  smooth solutions  have been studied in \cite{SAEZ}.  

In this paper we consider the particular case in which the initial set  $E_0$ can be written- up to suitable rotation of coordinates- as the subgraph of a function $u_0:\R^n\to \R$. 
So, it is possible to define the evolution  $E_t$ at time $t$  as the subgraph of the solution $u(\cdot, t)$ to the following 
nonlocal quasilinear system
\begin{equation}\label{levelset}\begin{cases}  u_t+\sqrt{1+|Du |^2}H_s((x,u(x,t)),  \{ (x',z)\ | z\leq u(x',t)\}) =0\\ 
u(x,0)=u_0(x).\end{cases}\end{equation}
 
\begin{theorem}\label{ex} 
Let $u_0:\R^n\to \R$ be a uniformly continuous function. \\ 
Let $v, w\in C(\R^n\times [0, +\infty))$ be respectively a viscosity subsolution and a viscosity supersolution to \eqref{levelset} such that $v(x,0)\leq u_0(x)\leq w(x,0)$.  
\\ Then $v(x,t)\leq w(x,t)$ for all $(x,t)\in \R^n\times (0, +\infty)$. 

In particular \eqref{levelset} 
admits a unique viscosity solution $u(x,t)\in C(\R^n\times [0, +\infty))$ with $u(x,0)=u_0(x)$. 
Moreover, if $u_0$ is Lipschitz continuous with Lipschitz constant $\|Du_0\|_\infty$, then $|u(x,t)-u(x',t)|\leq \|Du_0\|_\infty|x-x'|$ for all $x,x'\in \R^n$ and $t>0$. 

\end{theorem}

\begin{proof} 
For every  $x\in \R^n$ and $z\in\R$, we define the uniformly continuous function $U_0(x,z):= u_0(x)-z$,    and the functions $V(x,z,t)= v(x,t)-z$, $W(x,z,t)= w(x,t)-z$. Then it is easy to check that $V,W$ are respectively a viscosity sub and supersolution to \eqref{level1} 
  such that $V(x,z,0)\leq U_0(x,z)\leq W(x,z,0)$. 
 Then by the comparison principle proved in \cite{i} (and for general nonlocal geometric equations in \cite{cmp}) we get that $V(x,z,t)\leq W(x,z,t)$ for all $x\in\R^n$, $z\in \R$, $t>0$.  This implies the result. 
 
Moreover, again by the results proved in \cites{i, cmp},    the system \eqref{level1} admits a unique viscosity solution $U(x,z,t)$.  If $U_0$ is Lipschitz continuous, then it is easy to check that also $U(x,z,t)$ is Lipschitz continuous in space, with Lipschitz constant less or equal to the Lipschitz constant of $U_0$, by using the comparison principle and the invariance by translation of the differential operator appearing in \eqref{level1}.  By comparison and again using the fact that the operator is invariant by translation in space, for every $h\in\R$ 
\[ U(x,z,t)+h= U(x, z+h,t).\]  Therefore, we conclude that $U(x,z,t)= z-u(x,t)$, where $u$ is a viscosity solution to \eqref{levelset}. 
\end{proof}

We now define the rescaled time variables as follows: 
\begin{equation}\label{resc} \tau(t):=\frac{\log(t(s+1)+1)}{s+1}\qquad \text{ that is }\qquad t= \frac{e^{(s+1)\tau}-1}{s+1}. \end{equation}   
In the rescaled time variables, the evolution \eqref{kflow} becomes 
\begin{equation}\label{kflowresc} 
\partial_\tau \tilde p\cdot  \tilde\nu= -  \tilde p \cdot  \tilde \nu -  H_s(\tilde p,\tilde E_\tau).\end{equation} 
We define the rescaled variables $e^{\tau} y= x$, with $\tau$ as in \eqref{resc}, and set 
\begin{equation}\label{rescf} \tilde u(y,\tau):=  e^{- \tau} u \left(y e^{\tau},  \frac{e^{(s+1)\tau}-1}{s+1}\right).\end{equation} 
Then $\tilde u$ solves the system
\begin{equation}\label{levelsetresc}\begin{cases}  \tilde u_\tau+\tilde u -  D\tilde u\cdot y+\sqrt{1+|D\tilde u |^2}H_s((y, \tilde u(y, \tau)), \{ (y',z)\ | z\leq \tilde u(y',t)\}) =0\\ 
\tilde u(y,0)=u_0( y).\end{cases}
\end{equation}
 
 Clearly, 
 the same existence, uniqueness and regularity results stated for  \eqref{levelset} in Theorem \ref{ex}  are valid for  \eqref{levelsetresc}.

\subsection{Fractional curvature on graphs} 
 
%
%
%

We recall  an equivalent formulation of the fractional mean curvature  $H_s$ on  graphical hypersurfaces, see \cite{SAEZ, bfv}. First of all observe that if $\Pi=\{(x',z'),\ | \ z'\geq u(x,t)+Du(x,t)\cdot (x'-x)\}$ 
then by symmetry 
\[\lim_{\e\to 0}
\int_{\R^{n+1}\setminus B_\eps(p)}\Big(
\chi_{\R^{n+1}\setminus \Pi}(y)-\chi_{\Pi}(y)\Big)\ \frac{1}{|p-y|^{n+1+s}}\,dy=0.\]
Therefore, for  $p=(x, u(x,t))$  and $E_t:=\{ (x',z)\ |\  z\leq u(x',t)\}$,  we get (intending the integrals in the principal value sense)  \begin{eqnarray} \nonumber 
H_s(p,E_t) &=&  
\int_{\R^{n+1}}\left(\frac{
\chi_{\R^{n+1}\setminus E_t}(y)-\chi_{\Pi}(y)}{|p-y|^{n+1+s}}+\frac{  \chi_{\R^{n+1}\setminus \Pi}(y)-\chi_{E_t}(y)}{|p-y|^{n+1+s}}\right)\, dy  \\\nonumber 
&=& 2\int_{\R^n}   \int_{u(x',t)}^{ u(x,t)+Du(x,t)\cdot (x'-x)} \frac{1}{( |x'-x|^2+|z'-u(x,t)|^2)^{(n+1+s)/2}}\, dzdx'\\
&=& 2\int_{\R^n}\frac{1}{|x-x'|^{n+s}}\int_{\frac{u(x',t)-u(x,t)}{|x-x'|}}^{Du(x,t)\cdot\frac{x-x'}{|x-x'|}}\frac{1}{(1+w^2)^{(n+1+s)/2}}\,dwdx'.  \label{hs}\end{eqnarray} 
We now introduce the function
\[G_s(t):=\int_0^t \frac{1}{(1+w^2)^{(n+1+s)/2}}dw.\]
By \eqref{hs} we get 
  \begin{eqnarray}\nonumber H_s((x,u(x,t)), E_t)&= & 2  \int_{\R^n}\frac{ G_s\left(Du(x,t)\cdot\frac{x-x'}{|x-x'|}\right)-G_s\left(\frac{u(x',t)-u(x,t)}{|x-x'|}\right)}{|x-x'|^{n+s}} dx'\\
&=& 2  \int_{\R^n}\frac{ G_s\left(Du(x,t)\cdot\frac{z}{|z|}\right)-G_s\left(\frac{u(x-z,t)-u(x,t)}{|z|}\right)}{|z|^{n+s}} dz . \label{hsgraph}
\end{eqnarray} 
We observe that this formula holds also in the viscosity sense (that is it is verified at points where the graph of $u$ can be touched with paraboloids). 

Moreover, if we change variable from $z$ to $-z$ in \eqref{hsgraph} and recalling that $G_s$ is odd, we get 
  \begin{eqnarray}\label{hsgraph1}
H_s((x,u(x,t)), E_t)&= &-   \int_{\R^n}\frac{ G_s\left(\frac{u(x+z,t)-u(x,t)}{|z|}\right)+ G_s\left(\frac{u(x-z,t)-u(x,t)}{|z|}\right)}{|z|^{n+s}} dz \\
\nonumber &= &-   \int_{\R^n}\frac{ G_s\left(\frac{u(x+z,t)-u(x,t)}{|z|}\right)- G_s\left(\frac{u(x,t)- u(x-z,t) }{|z|}\right)}{|z|^{n+s}} dz   \\  &=&-   \int_{\R^n}  A(x,z,u)
   \frac{u(x+z,t)+u(x-z,t)-2u(x,t)}{|z|^{n+s+1}}dz.   \nonumber \end{eqnarray} 
where 
\begin{equation}\label{prova} A(x,z,u):=\int_0^1 G_s'\left(w \frac{u(x+z,t) -u(x,t)}{|z|}+  (1-w) \frac{u(x,t)-u(x-z,t)}{|z|}\right) dw. \end{equation} 
Observe that $A(x,z,u)=A(x,-z,u)$ and
\[   \left(1+4\|Du_0\|_\infty^2\right)^{-\frac{n+s+1}{2}}  \leq A(x,z,u)  \leq 1.  \]
This implies that the differential operator $H_s((x,u(x,t)), E_t)$ is elliptic,   see e.g. \cite{serra}.

\section{Regularity results}\label{sectionreg} 
In this section, we provide some regularity results for the flow starting from a Lipschitz graph. 
These results are based on the fact that the fractional curvature for graphs is an elliptic fractional operator, and so it enjoys regularizing effects. 

  \begin{proposition}\label{reg} 
 Let $u_0:\R^n\to \R$ be a Lipschitz continuous function. Then  there exists $\alpha\in (0,s)$ depending on $s$ and $n$ such that  the viscosity solution $u(x,t)$ to \eqref{levelset}  is in $C^{1+\alpha}(\R^n)$ for every $t>0$,  with norm uniformly bounded in   $ \R^n\times [t_0, +\infty)$ by a constant only depending on $\|Du_0\|_\infty$ and $t_0$.  
  In particular,   there exists a constant $C>0$ only depending  on $\|Du_0\|_\infty$ and $t_0$ such that 
  \[\|Du(\cdot, t)\|_{C^{\alpha}(\R^n)}\leq Ct^{-\frac{\alpha}{s+1}} \qquad \text{for }  t\geq t_0. \]
  
 \noindent This implies that, if $\tilde u$ is the rescaled function defined in \eqref{rescf}, for every $\tau_0>0$, 
  there exists a constant $C>0$ only depending on  $\|Du_0\|_\infty$ and $\tau_0$ such that
  \[\|\tilde Du(\cdot, \tau )\|_{C^{\alpha}(\R^n)}\leq C   \left(\frac{1-e^{-\tau_0(s+1)}}{s+1}\right)^{-\frac{\alpha}{s+1}} \qquad \text{for }  \tau\geq \tau_0. \]  
 
  \end{proposition} 
  \begin{proof} If $u_0$ is Lipschitz continuous, then by Theorem \ref{ex} the solution to \eqref{levelset} is Lipschitz continuous in $x$ with  
  Lipschitz constant  bounded by $\|Du_0\|_\infty$. Then the differential operator $H_s$ on graphical hypersurfaces is elliptic, see \eqref{hsgraph1}, see e.g. \cite{serra, silv} for the definition. Then, by applying the H\"older regularity theory to the incremental quotients of $u$ (see \cite[Theorem 2.1, Theorem 2.2]{serra}) we get that they are of class $C^\alpha$ for some $\alpha\in (0,1)$, with norm bounded by the Lipschitz constant of $u$. 
  
  Finally, observe that for every $r>0$, there holds that $v_r(x,t)= r^{-1}u(rx, r^{1+s}t)$ is the viscosity solution to \eqref{levelset} with initial datum $v_0(x)=r^{-1} u_0(x)$. Then $v_0$ is Lipschitz continuous, with the same Lipschitz constant as $u_0$, and we may apply to $v_r$ the same regularity results as for $u$. In particular for every $t_0$ there exists 
a constant $C$ depending on $t_0$ and on $\|Du_0\|_\infty$ such that for all $t\geq t_0$, 
$ \|D v_r(\cdot, t)\|_{C^{\alpha}} \leq C. $
Rescaling back to $u$,  for every $t\geq t_0 r^{s+1}$ we get
\[ \|D u(\cdot, t)\|_{C^{\alpha}} \leq C r^{-\alpha}. \]   
We conclude by choosing $r= (t/t_0)^{\frac{1}{s+1}}$. 

Finally we consider the rescaled solution $\tilde u$. Note that $D\tilde u(y,\tau)=Du\left( y e^\tau, \frac{e^{\tau(s+1)}-1}{s+1}\right)$ and 
$\|D\tilde u(\cdot, t)\|_{C^{\alpha}(\R^n)}= e^{\tau \alpha} \|D u(\cdot, \frac{e^{\tau(s+1)}-1}{s+1})\|_{C^{\alpha}(\R^n)}$. 
Then by the previous estimate,    for every $\tau\geq \tau_0$, 
 \[\|D\tilde u(\cdot, \tau )\|_{C^{\alpha}(\R^n)}\leq e^{\alpha \tau}  C\left(\frac{e^{\tau(s+1)}-1}{s+1}\right) ^{-\frac{\alpha}{s+1}} \leq 
 C \left(\frac{1-e^{-\tau_0(s+1)}}{s+1}\right)^{-\frac{\alpha}{s+1}}. \]
  \end{proof}
  
Let now $u(x,t)$ be a  $C^{2,1}$  solution to \eqref{levelset} with initial datum $u_0$, and 
 let \[w(x,t):= \sqrt{1+|Du(x,t)|^2}  H_s((x,u(x,t)), E_t).\] 
 Since $u_t(x,t)= -w(x,t)$, using \eqref{hsgraph1}
 we compute
 \begin{eqnarray*} \nonumber  w_t &=&  -\frac{H_s((x,u(x,t)), E_t)}{\sqrt{1+|Du|^2}} Du(x,t)\cdot Dw(x,t)  +\\
 && +2 \sqrt{1+|Du|^2} \int_{\R^n} G_s'\left(\frac{u(x+z,t)-u(x,t)}{|z|}\right) \frac{w(x+z, t)-w(x,t)}{|z|^{n+s+1}}dz.
 \end{eqnarray*} 
 Therefore $w$ is a solution to
  \begin{equation}\label{vhs2} w_t+ B(x,t) \cdot Dw(x,t)  -2  \int_{\R^n} C(x,z,t) \frac{w(x+z, t)-w(x,t)}{|z|^{n+s+1}}dz=0,
  \end{equation} 
 where 
 \begin{equation}\label{coeff} B(x,t):= \frac{H_s((x,u(x,t)), E_t)}{\sqrt{1+|Du(x,t)|^2}} Du(x,t)\qquad C(x,z,t):=\sqrt{1+|Du(x,t)|^2}G_s'\left(\frac{u(x+z,t)-u(x,t)}{|z|}\right).\end{equation}
 
  \begin{lemma}\label{propbound} Let $u_0:\R^n\to \R$ be a   Lipschitz continuous function such that
   $H_s((x,u_0(x)), E_0)$ is bounded, and let $u(x,t)$ be the viscosity solution to \eqref{levelset} with initial datum $u_0$.

   Then $C(x,z,t)$ is well defined for every $x, z\in \R^n, t>0$,  $C(\cdot,z,t)\in C^\alpha(\R^n)$ and 
  \[0<(1+\|Du_0\|^2_\infty)^{-\frac{n+s+1}{2}}\leq C(x,z,t)\leq \sqrt{1+\|Du_0\|^2_\infty}. \]
 Moreover, the following inequalities hold in the viscosity sense:
 \[ -\frac{1}{2}C\leq   \frac{H_s((x,u(x,t)), E_t)}{\sqrt{1+|Du(x,t)|^2}} |Du(x,t)| \leq \frac{1}{2}C, \]
 where $C=  \| H_s((x,u_0(x)), E_0)\|_\infty (1+\|Du_0\|_\infty)$. 
 \end{lemma} 
  
  \begin{proof} 
  First of all we observe that by Theorem \ref{ex},  $u(\cdot, t)$ is Lipschitz with $\|Du(x,t)\|_\infty\leq \|Du_0\|_\infty$. 
Then the regularity and the  bounds on $C(x,z,t)$ are a direct consequence of the definition of $G_s$. 
 
  Let $C:= \| H_s((x,u_0(x)), E_0)\|_\infty (1+\|Du_0\|_\infty)$. 
 Note that $u_0(x)\pm Ct$ are respectively  a supersolution and a subsolution to \eqref{levelset}, so that by comparison we get
 \[u_0(x)-Ct\leq u(x,t)\leq u_0(x)+Ct\qquad \text{for all $t\geq 0$.}
 \] 
 Moreover for every $t\ge \tau>0$, the functions $u(x, t)\pm \sup_x|u(x, \tau)- u_0(x)|$ are respectively a supersolution and a subsolution to \eqref{levelset} 
 with initial datum $u(x, \tau)$, whence
 $$|u(x, t+\tau)- u(x,t)|\leq \sup_x|u(x, \tau)- u_0(x)|\leq C\tau .$$ 
 This implies that $u(x, \cdot)$ is Lipschitz continuous with  $|u_t(x, t)|\leq C$,  which in turns implies that, in the viscosity sense,  
 \[ -C\leq \sqrt{1+|Du(x,t)|^2}  H_s((x,u(x,t)), E_t)\leq C\] for all $x\in \R^n$ and $t>0$.  We now conclude recalling that $u(\cdot, t)\in C^1$. 
  \end{proof}

\begin{theorem}\label{regufinale} 
Let $u_0:\R^n\to \R$ be a    Lipschitz continuous function such that   $H_s((x,u_0(x)), E_0)$ is bounded. Let $u(x,t)$ be the viscosity solution to \eqref{levelset} with initial datum $u_0$. 

Then $u\in C^\infty$,  with norms   bounded in $\R^n\times [t_0, +\infty)$ by constants   depending on $t_0, \|Du_0\|_\infty$ and $\|H_s((x,u_0(x)), E_0)\|_\infty$. 
 
Finally, the map $t\mapsto \sup_{x\in\R^n}\sqrt{1+|Du(x,t)|^2} |H_s((x,u(x,t)), E_t)|$ is decreasing  in $t$. 
\end{theorem}

\begin{proof} We  first prove the result under the additional assumption that $u_0$ is in $C^{2+\alpha}$ with bounded norms. Then  the general case will follow by the stability of viscosity solutions with respect to uniform convergence, recalling that all the estimates depend only on the Lipschitz constant of $u_0$ and on $\|H_s((x,u_0(x)), E_0)\|_\infty$.

The short time existence result in \cite[Theorem 5.1]{lamanna} implies that, if $u_0\in C^{2+\alpha}$ with bounded norms, then there exists a time $t>0$, such that the system 
\eqref{levelset} admits a smooth solution $u(x,t)$.  

Let  $w(x,t)$ be the unique   solution to \eqref{vhs2} with  
 initial datum $\sqrt{1+|Du_0(x)|^2} H_s((x,u_0(x)), E_0)$.  Since $u$ is smooth, by the computations in  \eqref{vhs2} and \eqref{coeff}, we get that
 \[w(x,t)= \sqrt{1+|Du(x,t)|^2} H_s((x,u(x,t)), E_t).\] 
  
By  comparison $w(x,t)$ is bounded by $\|w(x,0)\|_\infty$ and $\sup_x|w(x,t+r)|\leq \sup_x|w(x,r)|$, for every $r\geq 0$ and $t>0$.

 Since $w$ is a  bounded viscosity solution of a linear integro-differential equation,  with bounded drift and uniformly elliptic integro-differential operator, the regularity results obtained in \cite[Theorem 8.1]{silv} apply. Hence  there exists $\alpha\in (0,s)$ such that
$w(\cdot,t)\in C^{1+\alpha}(\R^n)$, $w(x, \cdot) \in C^{\frac{1+\alpha}{2}}(0, +\infty)$ with
\[\sup_{t\in (0, T)} \|w(\cdot, t)\|_{ C^{1+\alpha}}+\sup_{x\in\R^n} \|w(x,\cdot)\|_{ C^{\frac{1+\alpha}{2}}} \leq C(\|Du_0\|_\infty, \|H_s((x,u_0(x)), E_0)\|_\infty, s).\]
 This implies that  $u_t(x,\cdot)= \sqrt{1+|Du(x,\cdot)|^2} H_s((x,u(x,\cdot)), E_t)\in  C^{ \frac{1+\alpha}{2}}(0, +\infty)$. 
 
 Moreover
 since $\sqrt{1+|Du(x,t)|^2} H_s((x,u(x,t), E_t)\in C^{1+\alpha}(\R^n)$ as a function of $x$, recalling that $u\in C^{1+\alpha}$ 
 by Proposition \ref{reg}, with norm bounded only by $\|Du_0\|_\infty$, 
 by the bootstrap argument in \cite[Theorem 6]{bfv} we get that $u(\cdot, t)\in C^{1+s+\alpha+\beta}(\R^n)$ for all $\beta<1$ and all $t>0$, 
 with norm bounded only on $\|Du_0\|_\infty, \|H_s((x,u_0(x)), E_0)\|_\infty$.
 
 Finally, we apply the bootstrap regularity argument obtained in  \cite{lamanna} and we get the full regularity. 
  \end{proof} 
  
 \section{Homothetically expanding   graphical solutions}\label{sectionhom} 
We shall provide a complete characterization of graphical homothetically expanding solutions to \eqref{kflow}. 

\begin{theorem}\label{omogenee}Let $\bar u_0:\R^n\to \R$  be a 
 Lipschitz continuous and positively $1$-homogeneous function, that is,
\begin{equation}\label{hom}\exists C>0\ \  |\bar u_0(x)-\bar u_0(y)|\leq C|x-y|\qquad \bar u_0(rx)=r\bar u_0 (x)\qquad \forall r>0, x,y\in \R^n.\end{equation} 
Then, for every $t>0$ the subgraph  $\bar E_t$ of  the viscosity solution $\bar u(x,t)$ to  \eqref{levelset}, with initial datum  $\bar u_0$, satisfies  for $p\in \partial \bar E_t$ 
 \begin{equation}\label{self}  
 p\cdot \nu=-t(s+1)H_s(p  , \bar E_t),\end{equation}  
 that is,  for every $t>0$  $\bar E_{t}$ satisfies \eqref{exp} with $c^{-1}=t(s+1)$, and the flow starting from $E_t$  is an expanding homothetic solution to \eqref{kflow} 
 according to Definition \ref{defhom}. 
 
Moreover  $\bar u(x,t)$     is in $C^\infty(\R^n\times (0, +\infty))$, and  for every $T >0$, \[\lim_{t\to +\infty} |\bar u(x,t+T)-\bar u(x,t)|=0\qquad \text{       locally uniformly in $\R^n$. }\]
\end{theorem} 

\begin{proof}
By the fact that the differential operator  is invariant under translations and by uniqueness of  solutions, see Theorem \ref{ex},  we get that  for all $r\neq 0$, it holds
\[\bar u(x,t)= \frac{1}{r} \bar u(rx, r^{s+1} t).\] 
Letting $r:=t^{-\frac{1}{s+1}}$ for $t>0$, we get 
\begin{equation} \label{cono}\bar u(x,t)=t^{\frac{1}{s+1}}\bar u(xt^{-\frac{1}{s+1}}, 1). \end{equation} 
This implies that, if $p\in \partial \bar E_1$ then $p t^{\frac{1}{s+1}}\in \bar E_t$ and  \begin{equation}\label{curv} H_s(p t^{\frac{1}{s+1}},\bar E_t)=  t^{-\frac{s}{s+1}}H_s(p  , \bar E_1).\end{equation} 
Substituting in \eqref{kflow}  we get that $ \bar E_{1 }$ solves \eqref{exp} with $c^{-1}= s+1$.  The same argument holds substituting $t=1$ with another positive time $t$. 

By  the uniform $C^{1,\alpha}$ estimate in Proposition \ref{reg} we know that $\bar u(x,1)$   is in $C^{1, \alpha}$. Moreover, since $ \bar E_{1 }$ solves \eqref{exp} with $c^{-1}= s+1$, we get that $\bar u(x,1)$ solves
\begin{equation}\label{curvfun}\bar u (x,1)-  D\bar u(x,1)\cdot x+(1+s)\sqrt{1+|D\bar u(x,1)|^2}H_s((x, \bar u(x,1)),\bar E_1) =0.\end{equation}
Therefore, since   $\sqrt{1+|D\bar u(x,1)|^2}H_s(x, \bar u (x,1),\bar E_1)$ is in $C^{\alpha}$, with norm locally bounded by the Lipschitz constant of $\bar u_0$, we can apply the bootstrap argument in \cite[Theorem 6]{bfv} and get that $\bar u(x,1)$ is in $C^\infty$. Finally, since $\bar u(x, t)= t^{\frac{1}{s+1}}\bar u(xt^{-\frac{1}{s+1}}, 1)$ for every $t>0$, we conclude that $\bar u$ is in $C^\infty(\R^n\times (0, +\infty))$. 

Now, observe that by scaling properties \eqref{cono},  for every $T>0$ and $t>0$ and by the fact that $\bar u(x,t)$ is Lipschitz continuous with the same Lipschitz constant as  $\bar u_0$, 
\begin{eqnarray*} |\bar u(x,t+T)-\bar u(x,t)|&=&  |(t+T)^{\frac{1}{s+1}}\bar u(x(t+T)^{-\frac{1}{s+1}}, 1)-t^{\frac{1}{s+1}}\bar u(xt^{-\frac{1}{s+1}}, 1)| \\
&\leq& (t+T)^{\frac{1}{s+1}}|\bar u(x(t+T)^{-\frac{1}{s+1}}, 1) - \bar u(x t^{-\frac{1}{s+1}}, 1)| \\ && + [(t+T)^{\frac{1}{s+1}}-t^{\frac{1}{s+1}}| |\bar u(xt^{-\frac{1}{s+1}}, 1)|\\
&\leq & C (t+T)^{\frac{1}{s+1}}|x| |(t+T)^{-\frac{1}{s+1}}-t^{-\frac{1}{s+1}}| + C [(t+T)^{\frac{1}{s+1}}-t^{\frac{1}{s+1}}| |x| t^{-\frac{1}{s+1}}\\
&\leq & C|x| \left(\left(1+\frac{T}{t}\right)^{\frac{1}{s+1}}- \left(1-\frac{T}{t}\right)^{\frac{1}{s+1}}\right). 
\end{eqnarray*} Sending $t\to +\infty$, we get the result. 

  \end{proof} 
By using the properties of homothetically expanding solutions, we show the following result about uniform continuity of solutions to \eqref{levelset}.
\begin{proposition}\label{holder}  Let $u_0:\R^n\to \R$ be a Lipschitz continuous function. Then the viscosity solution to \eqref{levelset} with initial datum $u_0$ 
satisfies for all $x, y\in \R^n, t,s\geq 0$ 
\[|u(x,t)-u(y,r)|\leq \|Du_0\|_\infty |x-y|+   K |t-r|^{\frac{1}{s+1}}
 \]
for some constant $K>0$ which depends only on $\|Du_0\|_\infty$. 

\end{proposition} 
\begin{proof}
   We prove just the H\"older continuity in time, since the Lipschitz  continuity has already been proved in Theorem \ref{ex}.  Let $C=\|D u_0\|_\infty$.
 
Let  $v_0(x)=C|x|$.  Since $v_0$ satisfies \eqref{hom},   the solution $v_C(x,t)$ with initial datum $v_0$ is a homothetically expanding solution to \eqref{levelset}. Moreover, since  $H_s((x,C|x|), \{(x',z'), \ |\ z'\leq C|x'|\})\leq 0$  in the viscosity sense at every $x\in \R^n$,  $v_0(x)$ is a stationary subsolution to \eqref{levelset}, 
which implies by comparison that $v_C(x,t)\geq v_0(x)$ for every $t>0$, and then, again by comparison that $v_C(x,t)\geq v_C(x, r)$ if $0<r<t$.  

Let us fix $x_0\in \R$. By Lipschitz continuity we get that $u_0 (x+x_0)\leq v_0(x)+u_0(x_0)$, so by comparison we conclude that $  u(x+x_0,t)\leq v_C(x,t) + u_0(x_0)$, for all $x$. If we compute the previous inequality in $x=0$, using also \eqref{defu} we get
\[ u(x_0,t)- u_0(x_0)\leq v_C(0,t)= v_C(0,1) t^{\frac{1}{s+1}}\qquad \forall x_0\in\R^n.\]
By comparison this implies that
\[  u(x_0,t+r)-  u(x_0,t)\leq v_C(0,1) r^{\frac{1}{s+1}}\qquad \forall x_0\in\R^n.\]
The other inequality is obtained with an analogous argument by considering $-C|x|$.  \end{proof} 
  \begin{remark} We observe that if $\bar u_0$ in Theorem \ref{omogenee} is also  a convex function, then there exists $C>0$ such that   \begin{equation}\label{claim1} \sup_{x\in\R^n} |H_s(x, \bar u(x,t))|\leq C t^{-\frac{s}{s+1}}\qquad \forall t>0.\end{equation}   We conjecture that  this  property  is actually true also for nonconvex functions $\bar u_0$, which satisfy \eqref{hom}, as in the local setting.   
Observe that due to  \eqref{curv} it is sufficient to show that  that there exists $C>0$ such that  \[ \sup_{x\in\R^n} | H_s(x, \bar u (x,1))|\leq C,\] 
which is in turn equivalent by \eqref{curvfun} to show that  
\[\sup_{x\in \R^n}|\bar u(x,1)-D\bar u(x,1)\cdot x|\leq C(s+1)^{-1}.\]  

Note that $\bar u(x,1)$ is a convex function, since convexity is preserved by the fractional flow \eqref{kflow}, see \cite{cnr}. Moreover $\bar u(x,1)\geq \bar u_0(x)$ by comparison (since $\bar u_0(x)$ is a stationary subsolution to \eqref{levelset}).  For $\lambda>0$, and $x\in \R^n$ fixed, we define
$v(\lambda)= \bar u(\lambda x,1)$.  Note that this function is convex. By convexity we get that
\[ \bar u(0,1)=v(0)\geq v(1)-v'(1)=\bar u(x,1)-D\bar u(x,1)\cdot x.\]
Recalling Proposition \ref{holder} we get that there exists $K>0$ depending only on the Lipschitz constant of $\bar u_0(x)$ such that $\bar u_0(\lambda  x)+ K\geq \bar u(\lambda x,1)$ for every $\lambda\in \R$. So using again convexity of $v$ and by \eqref{hom} we get  for $\lambda\geq 0$, 
\[ \lambda \bar u_0( x)+ K=\bar u_0(\lambda  x)+ K\geq \bar u(\lambda x,1)=v(\lambda)\geq v(1)+v'(1)(\lambda-1)=\bar u(x,1)+(\lambda-1)D\bar u(x,1)\cdot x .\] 
This implies, sending $\lambda\to +\infty$,  that $\bar u_0(x)\geq D\bar u(x,1)\cdot x$ and so in turn
\[\bar u(x,1)-D\bar u(x,1)\cdot x\geq \bar u(x,1)-\bar u_0(x)\geq 0.\]
So, we proved that 
\[ 0\leq \bar u(x,1)-D\bar u(x,1)\cdot x\leq \bar u(0,1)\] which gives the result. 

  \end{remark} 
  \begin{remark}\label{convexcono}
Observe that if \eqref{claim1} is satisfied,  then we may strengthened the convergence result in Theorem \ref{hom}, that is 
  for every $T>0$ it holds  that \[\lim_{t\to +\infty} |\bar u(x,t+T)-\bar u(x,t)|=0\qquad \text{   uniformly in $\R^n$. }\] 
  Indeed substituting in the equation \eqref{levelset} and recalling that $\bar u(x,t)$ is uniformly Lipschitz, we get that $|\bar u_t(x,t)|\leq C't^{-\frac{s}{s+1}}$, for some $C'>0$ depending on $C$ and on the Lipschitz norm of $u_0$. So, integrating we get, for all $T>0$
\[|\bar u(x, t+T)-\bar u(x,t)|\leq C \left((t+T)^{\frac{1}{s+1}}-t^{\frac{1}{s+1}}\right) \]
and so for every $T>0$,  
\[\lim_{t\to +\infty} |\bar u(x,t+T)-\bar u(x,t)|=0\qquad \text{   uniformly in $\R^n$. }\]
\end{remark} 
 On the other hand, every homothetically expanding graphical solution to \eqref{kflow} is obtained as in Theorem \ref{omogenee}. 
\begin{proposition}\label{propchar} Assume that $E_1$ is a solution to \eqref{exp}  and that is the subgraph of  a Lipschitz continuous function $u_1:\R^n \to \R$.

Then  the solution $u(x,t)$ to \eqref{levelset} with initial datum $u(x,1)=u_1(x)$ is defined in $\R^n\times(t_0, +\infty)$ where $t_0=1-\frac{1}{c(s+1)}$ and  satisfies 
\begin{equation}\label{conv} \lim_{t\to t_0^+} u(x,t)= \bar u(x)\qquad \text{locally uniformly in $x$,} \end{equation}
where  $\bar u:\R^n\to\R^n$ is Lipschitz continuous and $1$-homogeneous as in  \eqref{hom}.
 \end{proposition}  

\begin{proof} 
 According to Definition \ref{defhom}, 
\[  E_t= \left[c(s+1)(t-1)+1\right]^{\frac{1}{s+1}} E_1=\lambda(t)E_1\]   is a solution to \eqref{kflow}, which means in particular that  the solution $u$ to \eqref{levelset} with initial datum $u(x,1)=u_1(x)$ can be obtained as 
\begin{equation}\label{defu} u(x,t):=\lambda(t)u_1\left(\frac{x}{\lambda(t)}\right)=\left[c(s+1)(t-1)+1\right]^{\frac{1}{s+1}} u_1\left(x \left[c(s+1)(t-1)+1\right]^{-\frac{1}{s+1}}\right).  \end{equation}
This implies immediately that $u$ is well defined in $\R^n\times (t_0, +\infty)$, where $t_0=1-\frac{1}{c(s+1)}$.  

Since $u_1$ is Lipschitz continuous, we get that $v_r(x)= ru_1(x/r)$ are equilipschitz, and moreover $|v_r(x)|\leq  r|u_1(0)|+\|Du_1\|_\infty |x|$. 
Then, by Ascoli Arzel\`a theorem,  up to subsequences there exist the limits   $\lim_{r_n\to 0^+} v_{r_n}(x)$, locally uniformly in $x$.  We claim that  actually the limit is unique, that is $\bar u(x):=\lim_{r\to 0^+} ru_1(x/r)$ locally uniformly in $x$.  If the claim is true, then it is easy to check that $\bar u(x)$ satisfies \eqref{hom} and moreover by \eqref{defu},   $\lim_{t\to t_0^+}u(x, t)=\lim_{t\to t_0^+} \lambda(t)u_1\left(\frac{x}{\lambda(t)}\right)=\lim_{r\to 0^+} ru_1(x/r)=\bar u(x)$. 

To prove the claim we observe that by \eqref{defu} $v_r(x) = ru_1(x/r)= u(x, \lambda^{-1}(r))$ for every $r>0$. 
Let $r_n\to 0$ such that $\lim_{r_n\to 0} v_{r_n}(x)= \bar v(x)$. 
Then   $u_{n}(x,t):= u(x, t+\lambda^{-1}(r_n))$  is the viscosity solution to \eqref{levelset} with $u_n(x,0)= v_{r_n}(x)$.  By stability of viscosity solutions with respect to uniform convergence, since $v_{r_n}\to \bar v$, we get that $u_n(x,t)\to \bar v(x,t)$ locally uniformly, where $\bar v(x,t)$ is the solution to \eqref{levelset} with initial datum $\bar v$. But actually $\bar v(x,t)= u(x, t+t_0)$ for every $t>0$, and then the limit $\bar v$ is unique and independent of the subsequence. 
 \end{proof}

\section{Convergence to  self-similar solutions} \label{sectionconv} 
We show that homotetically expanding solutions are the long-time attractors for the flow of Lipschitz graphs, when  the initial datum is a sublinear  perturbation of a $1$-homogeneous function.

We now generalize to the fractional curvature flows the result obtained in \cite{eh} for the mean curvature flow, under the  assumption 
 that the initial datum $u_0$ is Lipschitz continuous and is straight at infinity in the following sense: 
 There exists $K>0$ and $\delta>0$ such that
 \[|u_0(x)-Du_0(x)\cdot x|\leq K(1+|x|)^{1-\delta}\qquad \forall x\in \R^n. \] 
 This condition is equivalent to the following one: there exist a function $\bar u_0$ which satisfies \eqref{hom} and 
 constants $K>0$, $\delta>0$ such that 
\begin{equation}\label{straight} |u_0(x)- \bar u_0(x)|\leq K(1+|x|)^{1-\delta}. \end{equation} 

  Note that, if we define $\phi(r, x):=ru_0\left(\frac{x}{r}\right)$ for all $x\in \R^n$, then the straight at infinity condition reads, for $r\leq 1$, 
 \[\left|\frac{\partial }{\partial r} \phi(r,x)\right|= \left|u_0\left(\frac{x}{r}\right)- \frac{x}{r} \cdot Du_0\left(\frac{x}{r}\right)\right|\leq K(r +|x|)^{1-\delta} r^{\delta-1}\leq K(1+|x|)^{1-\delta}r^{\delta-1}.\]
 Then, for all $0< r_1<r_2\leq 1$ we get 
 \[\left| r_2u_0\left(\frac{x}{r_2}\right)-r_1u_0\left(\frac{x}{r_1}\right)\right|=\left|\int_{r_1}^{r_2} \frac{\partial }{\partial r} \phi(r,x)dr\right|  \leq  \frac{K}{\delta} (1+|x|)^{1-\delta} (r_2^{\delta}-r_1^\delta). \] 
  Observe that, since $u_0$ is Lipschitz continuous, up to a subsequence there exists the limit $\lim_{r\to 0^+} ru_0\left(\frac{x}{r}\right) $, 
  which is locally uniform in $x$.  
  By the previous inequality, we conclude that the limit is unique,   so that the limit is a function $\bar u_0$ which satisfies \eqref{hom} and finally  \[\left|\int_0^1 \frac{\partial }{\partial r} \phi(r,x)dr\right | =  |u_0(x)-  \bar u_0(x)| \leq  \frac{K}{\delta} (1+|x|)^{1-\delta} .\]   
 
 Actually, the convergence result proved in \cite{eh} is stronger than ours, since they provide exponential in time convergence of the flows. 

   \begin{theorem}\label{convthm2}  
   Let $u_0$ be a Lipschitz continuous function,  such that there exist  $\bar u$ which satisfies \eqref{hom}, and constants $K>0$, $\delta\in (0,1)$  for which \eqref{straight} holds.  
   Let  $u$ and  $\tilde u$ be respectively the solutions   to \eqref{levelset} and \eqref{levelsetresc} with  initial datum $u_0$, and $\bar u(x,t)$ be the  solution to \eqref{levelset} with initial datum $\bar u_0(x)$. Then  
 \[\lim_{\tau\to +\infty} \tilde u(y,\tau)= \bar u\left(y,\frac{1}{s+1}\right) \qquad \text{ locally uniformly in $C^1(\R^n)$}.\] In particular,  the rescaled flow \begin{equation}\label{rescev}  \frac{1}{[(s+1)t+1]^{1/(s+1)}} E_t \end{equation}  where $E_t$  is the subgraph of $u(\cdot, t)$,
converges  as $\tau\to +\infty$ to a graphical hypersurface $\tilde E$ which satisfies \eqref{exp} (with $c=1$).\end{theorem}
\begin{proof} 
Let $\psi:(0, +\infty)\to (0, +\infty)$ be a smooth function  such that $\psi(k)\equiv 0$ if $k<1$ and $\psi(k)\equiv 1$ if $k>2$.  Define  for $r>1$
\[u_0^r(x)=\bar u_0(x)+ \psi\left(\frac{|x|}{r}\right)(u_0(x)-\bar  u_0(x)) . \] 
Then our assumption implies that 
\[|u_0(x)- u_0^r(x)| \leq K(1+2r)^{1-\delta}\qquad \forall x\in \R^n. \] 
By the comparison principle we get that, if $\tilde u$ and $\tilde u^r$ are respectively the solutions to \eqref{levelsetresc} with initial datum $u_0$ and $u_0^r$, 
\begin{equation}\label{r1}  \tilde u^r(y,\tau)-K(1+2r)^{1-\delta}e^{-\tau} \leq \tilde u(y, \tau)\leq \tilde u^r(y,\tau)+ K(1+2r)^{1-\delta}e^{-\tau}\qquad \forall y\in \R^n, \tau>0.
\end{equation} 
On the other hand,  
\[\bar u_0(x)- \frac{2K}{ r^\delta}  |x|\leq u_0^r(x)\leq \bar u_0(x)+ \frac{2K}{r^\delta} |x|.\]   
Let $\bar u_{\pm r}$ be  the solutions to \eqref{levelset}  with initial datum respectively $\bar u_0(x)\pm \frac{2K}{ r^\delta}  |x|$. Note that $\bar u_0(x)\pm \frac{2K}{ r^\delta}  |x|$ satisfy \eqref{hom} with Lipschitz constant  $\|D\bar u_0\|_\infty+ \frac{2K}{ r^\delta}\leq \|D\bar u_0\|_\infty+2K$.  By scaling properties of  $\bar u_0(x)\pm \frac{2K}{ r^\delta}  |x|$, see \eqref{cono}, and by formula \eqref{rescf}, we get that  
$\bar u_{\pm r} \left(y, \frac{1-e^{-(s+1)\tau}}{s+1}\right)$ is the solution to \eqref{levelsetresc} with initial datum $\bar u_0(x)\pm \frac{2K}{ r^\delta}  |x|$. 
Then by comparison principle 
\[\bar u_{-r} \left(y, \frac{1-e^{-(s+1)\tau}}{s+1}\right)\leq \tilde u^r(y, \tau) \leq\bar u_{+r} \left(y, \frac{1-e^{-(s+1)\tau}}{s+1}\right). \] 
 By Proposition \ref{holder}, recalling that $\bar u_0(x)\pm \frac{2K}{ r^\delta}  |x|$ are Lipschitz functions with Lipschitz constant less than   $ \|D\bar u_0\|_\infty+2K$, we get that there exists $B$ depending only on $\|D\bar u_0\|_\infty$ and $K$ such that 
\[\bar u_{-r}\left(y, \frac{1}{s+1} \right)-Be^{-\tau} \leq \tilde u^r(y, \tau) \leq \bar u_{+r}\left(y, \frac{1}{s+1} \right)+Be^{-\tau}. \] 
Therefore by \eqref{r1} we conclude that for all $y\in \R^n$, $\tau>0$, and all $r>>1$, 
\[  \bar u_{-r}\left(y, \frac{1}{s+1} \right)-(B+K(1+2r)^{1-\delta})e^{-\tau} \leq \tilde u(y, \tau)\leq \bar u_{+r}\left(y, \frac{1}{s+1} \right)+(B+K(1+2r)^{1-\delta})e^{-\tau}. \] 
Note that as $r\to +\infty$,  $\bar u_{\pm r}\left(y, \frac{1}{s+1} \right)\to  \bar u \left(y, \frac{1}{s+1} \right)$ locally uniformly in $y$ by stability of viscosity solutions, 
since $\bar u_0(x)\pm \frac{2K}{r^\delta} |x|\to \bar u_0(x)$ locally uniformly.
 So taking $r=e^\tau$ in the previous inequality and sending $\tau\to +\infty$, we get the local uniform convergence of $\tilde u$. 
 Finally, since by Proposition \ref{reg}, $\tilde u$ has uniform $C^{1, \alpha}$ norm in $\R^n\times [t_0, +\infty)$, for every $t>0$, we conclude 
 that the locally uniform convergence holds in $C^{1,\alpha}$ sense. 
\end{proof}  

 \begin{remark}\upshape
If condition \eqref{straight}  is violated, then in general  we cannot expect the asymptotic convergence result proved  in Theorem \ref{convthm2}. Indeed,  observe that if $\tilde u$ is the solution to \eqref{levelsetresc} with initial datum $u_0$, then reasoning as in Proposition \ref{holder} we get that
\[|\tilde u(y,\tau)- e^{-\tau} u_0(ye^\tau)|\leq K \left(\frac{1-e^{-(s+1)\tau}}{s+1}\right)^{\frac{1}{s+1}}
\le K \left(\frac{1}{s+1}\right)^{\frac{1}{s+1}}=: K'.\]
In particular, if the convergence takes place, then for every compact set $B\subset\R^n$ there exists $\tau_B>0$ such that 
\[|e^{-\tau_1} u_0(ye^{\tau_1})- e^{-\tau_2} u_0(ye^{\tau_2})|\leq 3 K'\qquad \text{for all $\tau_1,\tau_2\ge \tau_B$.}\]
For instance, this condition is not satisfied by initial data oscillating at infinity between different positively homogeneous functions. 
We refer to  \cite[Proposition 6.1]{eh} for an explicit example. 
\end{remark} 

 \section{Convergence of the unrescaled flow}  \label{unrescaled} 
 
 In this section we will consider some cases in which  convergence of the unrescaled flow holds. To get stability without rescaling, we have to impose some decay or periodicity condition of the initial datum.  
 
  \subsection{Stability of hyperplanes} 
We show that hyperplanes are stable with respect to the flow \eqref{kflow}, that is, if the initial datum is flat at infinity (resp. periodic), then the solution stabilizes  to the hyperplane  at which the initial datum is  (resp. stabilizes to a constant).  

We remark that the behavior of the solution to \eqref{levelset} for these families of initial data is analogous to the behaviour of solutions to the fractional heat equation $u_t+ (-\Delta)^{\frac{s+1}{2}}u=0$, with the same initial data.  Analogous results for the local mean curvature flow of graphs have been obtained in \cite{e,clutt, n}, with different approaches: either  comparison with large balls as in our case (even if in the local case the argument is more involved), or reduction to stabilization of solutions to the heat equation. 

We start with a result about periodic initial data, showing that the solution stabilizes to a constant. For a particular class of periodic initial datum we may show that actually this constant is given by the mean value of the initial datum. 
   \begin{proposition}\label{periodic}Let $u_0:\R^n\to \R$ be a  Lipschitz function which is $\Z^n$ periodic. 
Then, there exists a constant $c\in (\min u_0, \max u_0)$ such that  the solution  $u$ to \eqref{levelset} with initial datum $u_0$ satisfies 
\[\lim_{t\to +\infty} u(x, t) =c\qquad\text{   uniformly in $C^1(\R^n)$. }\] 

If moreover  $u_0$  has  also the property that \begin{equation}\label{seno} 
\text{ there exists $v\in \R^n$ such that  for all  $x\in\R^n$, $u_0(x+v)=-u_0(x)$}\end{equation} then
\[\lim_{t\to +\infty} u(x, t) =0=\int_{[0,1]^n}u_0(x)dx\qquad\text{uniformly in $C^1(\R^n)$. }\]  \end{proposition} 
 \begin{proof} 
  We observe that by uniqueness the solution $u(x,t)$ is $\Z^n$ periodic in the $x$ variable. 
We define $M(t)=\max_x u(x,t)$ and $m(t)=\min_x u(x,t)$. Note that by Proposition \ref{holder}, $M(t)$ and $m(t)$ are H\"older continuous functions, and moreover by comparison, we have that  $\min_x u_0(x)\leq m(t)\leq M(t)\leq \max_x u_0(x)$, and that $M(t)$ is decreasing and $m(t)$ is increasing. Therefore $M(t)$ and $m(t)$ are differentiable a.e. 
We want to prove that $\lim_{t\to +\infty} M(t)-m(t)=0$. If this holds, then the result follows, recalling that $M(t)$ is decreasing and $m(t)$ is increasing. The $C^1$ convergence is a consequence of the uniform estimates in Proposition \ref{reg}. Assume by contradiction that $\lim_{t\to +\infty} M(t)-m(t)=\bar C>0$.
 We fix $t>0$ and $x_M\in  \text{argmax } u(\cdot,  t)\cap [0,1]^n$ and 
 $x_m\in  \text{argmin } u(\cdot,  t)\cap [0,1]^n$. Let $C(t)=M(t)-m(t)\geq \bar C$. We recall that $u(\cdot, t)$ is Lipschitz continuous with Lipschitz constant less than $\|Du_0\|_\infty$ and 
 we fix $\delta>0$ such that $\delta \|Du_0\|_\infty\leq \frac{\bar C}{2}$. It is immediate to check that  
 \[ u(x_M, t)-u(x,t)\geq u(x_M,t)-u(x_m,t)- \delta \|Du_0\|_\infty=C(t)-\frac{\bar C}{2}\geq \frac{C(t)}{2} \quad \forall x\in B(x_m, \delta),\] 
 and analogously \[u(x, t)-u(x_m,t)\geq \frac{C(t)}{2} \quad \forall x\in B(x_M, \delta).\] 
 Then, at every point of differentiability $t$, the functions  $M, m$ satisfy \[M'(t)= u_t(x_M, t), \ m'(t)=u_t(x_m,t),\qquad\text{ for all  $x_M\in\text{argmax } u(\cdot,  t), x_m\in\text{argmin } u(\cdot,  t)$.   }\]
Using the equation we get 
 \begin{eqnarray*} M'(t)=u_t(x_M, t)&=& -H_s(x_m, u(x_M,t))   \leq \left(1+\|Du_0\|_\infty^2\right)^{-\frac{n+s+1}{2}} \int_{\R^n}  
   \frac{u(y,t)-u(x_M,t)}{|y-x_M|^{n+s+1}}dz\\
   &\leq & - \left(1+\|Du_0\|_\infty^2\right)^{-\frac{n+s+1}{2}}\frac{C(t)}{2} \int_{B(x_m, \delta)} \frac{1}{|y-x_M|^{n+s+1}}dz\\
   &=& - \left(1+\|Du_0\|_\infty^2\right)^{-\frac{n+s+1}{2}}\frac{C(t)}{2}\frac{\omega_n\delta^n }{(\delta+1)^{n+s+1}}<0\end{eqnarray*} 
and 
\begin{eqnarray*} m'(t)=u_t(x_m, t)&=& -H_s(x_m, u(x_m,t))  \geq  \left(1+\|Du_0\|_\infty^2\right)^{-\frac{n+s+1}{2}}  \int_{\R^n}  
   \frac{u(y,t)-u(x_m,t)}{|y-x_m|^{n+s+1}}dz\\
     &\geq&  \left(1+\|Du_0\|_\infty^2\right)^{-\frac{n+s+1}{2}}\frac{C(t)}{2}\frac{\omega_n\delta^n }{(\delta+1)^{n+s+1}}>0.\end{eqnarray*} 
These two inequalities imply that  $M(t)$ is strictly decreasing, $m(t)$ is strictly increasing and $C'(t)\leq -K C(t)$ for a constant $K$ depending only on $\bar C$ and $\|Du_0\|_\infty$. Therefore
$\lim_{t\to +\infty}C(t)=0$, which is in contradiction with our assumption.

 Finally, observe that if $u_0 $ satisfies \eqref{seno}, then by uniqueness, $-u(x,t)=u(x+v, t)$. This implies that necessarily $\lim_{t\to +\infty} u(x,t)=-\lim_{t\to +\infty} u(x,t)$ and then the limit is $0$. 
 \end{proof} 
 
We first of all prove  stability of  constant functions in $\R^n$. 
\begin{theorem} \label{convstat}  Let  $u_0:\R^n\to \R$ be a  Lipschitz function such that 
\[\lim_{|x|\to +\infty} u_0(x ) =0.\]
Then,  the solution  $u$ to \eqref{levelset} with initial datum $u_0$ satisfies 
\[\lim_{t\to +\infty} u(x, t) =0\qquad \text{ uniformly in $C^1(\R^n)$. }\] 
\end{theorem}
\begin{proof}
First of all we observe that  it is sufficient to prove  the result for initial data which are nonnegative everywhere or nonpositive everywhere. Indeed the general case is easily obtained by using as barriers  the solutions with initial data $u_0^+=\max (u_0, 0)$ and $u_0^-=\min(u_0, 0)$.  

So, we prove the result only for the case  $u_0\geq 0$, since the other case $u_0(x)\leq	 0$ is completely analogous. 
Note that by comparison, since the constant are stationary solutions to \eqref{levelset}, $0\leq u(x,t)\leq \max_{y} u_0(y)$ for all $x\in \R^n, t>0$. 

We claim now that for  all $t>0$, \[\text{$\inf_x u(x,t)=0$  and that $M(t):=\max_x u(x,t)$   is decreasing in $t$. }\]
Indeed for every $\eps>0$, let us fix $R>0$ such that $|u_0(x)|\leq \eps$ for all $|x|\geq R$. For every $|x|>R$,  fix $K=|x|-R>0$ and observe that the ball $B((x, K+\eps), K)$ of center $(x, K+\eps)$ and radius $K$, is contained in $\R^{n+1}\setminus E_0$. By monotoniciy of the flow \eqref{kflow} with respect to inclusions, see \cite{cmp}, there holds that \[B((x,K+\eps), K(t))\subseteq \R^{n+1}\setminus E_t,\] where $K(t)=(K^{s+1}-(s+1)\bar c t)^{\frac{1}{s+1}}$, and $\bar c$ is the fractional mean curvature of the unit ball in $R^{n+1}$. Therefore, we get that  for all $t$ with $t\leq T\leq \frac{K^{s+1}}{2(s+1)\bar c}$ there holds 
\[u(x,t)\leq \eps + K-K(t)\leq \eps-K'(T) t=\eps+\frac{\bar c}{ (K^{s+1}-(s+1)\bar c T)^{\frac{s}{s+1}}} t\leq \eps+ \frac{\bar c 2^{\frac{s}{s+1}}}{K^s}   t. \]

This implies that  for all $\eps>0$, there exists $R>0$ such that for $|x|>R$, \begin{equation}\label{ball}  0\leq u(x,t)\leq \eps+\frac{\bar c 2^{\frac{s}{s+1}}}{(|x|-R)^s}   t\qquad \text{for all $t\leq \frac{(|x|-R)^{s+1}}{2(s+1)\bar c}$.}\end{equation}
This implies  that $\inf_x u(x,t)=0$ for all $t>0$ and moreover that $\sup_x u(x,t)=\max_x u(x,t)$. The fact that $\max_x u(x,t)$ is decreasing in $t$ is a consequence of comparison with stationary solutions.  
 
Now, since $M(t)$ is decreasing, let $\bar M=\lim_{t\to +\infty} M(t)=\inf_t M(t)$. We claim that $\bar M=0$. If the claim holds, then we get the conclusion.   The $C^1$ convergence is a consequence of the uniform estimates in Proposition \ref{reg}. 
 
Assume by contradiction that $\bar M>0$. We fix $0<\eps<\frac{\bar M}{2}$ and $\bar t>0$ such that $M(\bar t)\leq \bar M+\eps$. We fix also $R=R(\bar t)$ such that $u(x, \bar t)<\frac{\bar M}{2}$ for all  $|x|>R$.   

Now we aim to get a contradiction by constructing a  periodic barrier which satisfies (up to suitable vertical  translation) a condition like \eqref{seno}. 
We fix a Lipschitz continuous function $\phi:[-R, 2R] \to\R$, such that  $\phi$ is non increasing, $\phi(z)=\bar M+\eps$ for $z\in[-R, R]$, and $\phi(2R)=\frac{\bar M}{2}$. Now we extend it to a function $\phi:[-R, 5R]\to \R$ by putting $\phi(z)=\frac{3}{2}\bar M+\eps-\phi(z-3R)$   for all $2R\leq z\leq 5R$. Finally, we extend it by periodicity to be a $6R\Z$ periodic function. 
Then the function $v_0(x)=\phi(x\cdot e_1)$, is a $6R\Z^n$ periodic function, which is Lipschitz continuous, and satisfies  $v_0(x+3Re_1)=\frac{3}{2}\bar M+\eps-v_0(x)$.

Note that by construction, $u(x, \bar t)\leq v_0(x)$ for all $x\in \R^n$ and then by comparison \[u(x,t+\bar t)\leq v(x,t),\qquad \text{and in particular }  \limsup_{t\to +\infty} u(x,t)\leq \lim_{t\to +\infty} v(x,t)\] where $v(x,t)$ is the solution to \eqref{levelset} with initial datum $v_0$. Now by Proposition \ref{periodic} we get that $\lim_{t\to +\infty}v(x,t)=c$  uniformly in $C^1$, and moreover, since $v_0(x+3Re_1)=\frac{3}{2}\bar M+\eps-v_0(x)$ 
 there holds that $c=\frac{3}{2}\bar M+\eps-c$, and so $c=\frac{3}{4}\bar M+\frac{\eps}{2}<\bar M$, recalling our choice of $\eps$. 
But then we get that  $\limsup_{t\to +\infty} u(x,t)<\bar M$, in contradiction  with the definition of $\bar M$. 
 \end{proof} 
\begin{remark}\label{uniform}Let $u_0^\lambda$ be a family of Lipschitz continuous functions which fulfills uniformly in $\lambda$ the condition in Theorem \ref{convstat}, in the sense that
\[\sup_\lambda\sup_{|x|>R} |u_0^\lambda(x)|\to 0\qquad \text{ as $R\to +\infty$}.\]
Then it is easy to check that the convergence is uniform in $\lambda$ in the sense that  
\[\sup_\lambda \|u^\lambda(x,t)\|_{C^{1}}\to 0\qquad \text{ as $t\to +\infty$}\]
where $u^\lambda$ is the solution to \eqref{levelset} starting from $u_0^\lambda$. 
\end{remark} 
Finally we give the general result about stability of hyperplanes. We denote with $d(A,B)$  the Hausdorff distance between the sets $A,B$. 
\begin{corollary}\label{stabilitafinale} Let $E_0\subseteq \R^{n+1}$ be a set such that  $\partial E_0$ is a Lipschitz surface and that there exists a half-space $H$ for which
\[\lim_{R\to +\infty} d(E_0\setminus B(0, R), H\setminus B(0, R))=0. \]Then the outer and inner level set flows $E^+(t), E^-(t)$ defined in \eqref{outin} satisfy  
\[\lim_{t\to +\infty}d(E^+(t), H)=0= \lim_{t\to +\infty}d(E^-(t),  H ). \]\end{corollary} \begin{proof} Since the fractional mean curvature is invariant by rotations and translations,  we may assume without loss of generality that $H=\{(x, z)\in \R^{n}\times\R\ |\ z\leq 0\}$. Moreover, by the assumption that $\lim_{R\to +\infty} d(E_0\setminus B(0, R), H\setminus B(0, R))=0$, there exist two Lipschitz functions $u_0, v_0:\R^n\to\R$ such that $\lim_{|x|\to +\infty} u_0(x)=0=\lim_{|x|\to +\infty} v_0(x)$ and  $\{(x, z)\in \R^{n}\times\R\ |\ z\leq u_0(x)\}\subseteq E_0\subseteq \{(x, z)\in \R^{n}\times\R\ |\ z\leq v_0(x)\}$. By  comparison we get that 
$\{(x, z)\in \R^{n}\times\R\ |\ z\leq u(x,t)\}\subseteq E_t^-\subseteq E_t^+\subseteq \{(x, z)\in \R^{n}\times\R\ |\ z\leq v(x,t)\}$, where $u(x,t), v(x,t)$ are the solutions to \eqref{levelset} with initial datum $u_0, v_0$. 
By  Theorem \ref{convstat} $\lim_{t\to+\infty} u(x,t)=\lim_{t\to+\infty} v(x,t)=0$ uniformly in $\R^n$, and this gives the conclusion. 
\end{proof} 
\subsection{Stability of convex cones}
 In this section we provide   the convergence of the unrescaled flow in the case the initial data is decaying  at infinity to a  $H_s$-mean convex cone, staying above it. The result can be strenghtened if the initial cone is convex, by using the stability of hyperplanes. 
 \begin{proposition} \label{sconvex} Let $u_0:\R^n\to \R$ be a Lipschitz continuous function. Assume there exists  a non linear function $\bar u_0$ which satisfies \eqref{hom},  and  \begin{equation}\label{convex}
 \text{  $H_s(x, \bar u_0(x))\leq 0$ in the viscosity sense,}\end{equation}  
 such that  
 \[u_0(x)\geq \bar u_0(x)\qquad\text{and}\qquad  \lim_{|x|\to +\infty} u_0(x)-\bar u_0(x)=0. \] 
Then, if $u$ is  the solution to \eqref{levelset} with initial datum $u_0$, it holds
\[\lim_{t\to +\infty} u(x, t)-\bar u(x,t)=0\qquad\text{locally uniformly in $C^1(\R^n)$. }\] 
\end{proposition}

\begin{proof} 
Observe that by \eqref{convex},  $\bar u_0(x)$ is a stationary viscosity subsolution to \eqref{levelset}, therefore $\bar u(x, t)\geq \bar u_0(x)$ and so in particular $ \bar u(0, t)\geq \bar u_0(0)$. Observe that, if $\bar u_0$ is a homogeneous Lipschitz function, then either it is linear or it is singular at $0$, in the sense that the curvature in a neighborhood of $x=0$ is not bounded.  Therefore, since we assumed that $\bar u_0$ is non linear, then $\bar u(0,t)>0$  
since $\bar u$ is smooth by Theorem \ref{omogenee}.  
Again by comparison, we get that also $\bar u(x, t+r)\geq \bar u(x,t) $ for all $t\geq 0$, $r>0$, $x\in\R^n$.   

Fix $\eps>0$ and $R>0$ such that $u_0(x)\leq \bar u_0(x)+\eps$ for all $|x|>R$. Therefore  we get that
for all $T>0$,  \begin{equation}\label{due}u_0(x)\leq \bar u(x,T)+\eps\qquad \text{ for all $|x|>R$}.\end{equation} 
Observe now that  by \eqref{cono} and Lipschitz continuity \[\bar u(x,t)\geq \bar u(0, t)- C|x|= t^{\frac{1}{s+1}}\bar u(0,1)-C|x|.\] 
Since $\bar u(0,1)>0$ 
 there exists $T>0$ sufficiently large such that \begin{equation}\label{uno}\bar u(x,T)\geq T^{\frac{1}{s+1}}\bar u(0,1)-C|x|\geq u_0(x)\qquad\text{ for all $|x|\leq R$. }\end{equation}

Therefore, by \eqref{due}, \eqref{uno}, and by comparison we get that for some $T>0$ sufficiently large
\[u(x,t)\leq \bar u(x, T+t)+\eps\qquad \forall t\geq 0, x\in \R^n.\]

Note that since $u_0\geq \bar u$,  by comparison $u(x,t)\geq \bar u(x,t)$ for all $x,t$.
Then we get, for $\eps>0$ and $T>0$ fixed (and depending on $\eps$), 
\[0\leq u(x,t)- \bar u(x,t)\leq  \bar u(x, T+t)-\bar u(x,t)+\eps.\] 
We conclude by letting $t\to +\infty$ and recalling  that, by Theorem \ref{omogenee}, $\bar u(x, T+t)-\bar u(x,t)\to 0$ as $t\to +\infty$  uniformly in $x$, for all fixed $T$. 
\end{proof}

\begin{theorem}\label{convexthm}  Let $u_0:\R^n\to \R$ be a Lipschitz continuous function. Assume there exists a convex function $\bar u_0$ which satisfies \eqref{hom} and  such that  
\[ \lim_{|x|\to +\infty} u_0(x)-\bar u_0(x)=0. \] 
  Then, if $u$ the solution to \eqref{levelset} with initial datum $u_0$, 
\[\lim_{t\to +\infty} u(x, t)-\bar u(x,t)=0\qquad\text{  uniformly in $C^1(\R^n)$. }\] 
\end{theorem} 
\begin{proof}
We divide the proof in several steps.

\noindent
{\bf Step 1: for every $\eps>0$ there exists $T=T(\eps)$ such that  \begin{equation}\label{tre}u(x,t)\leq \bar u(x, T(\eps)+t)+\eps\qquad \forall t\geq 0, x\in \R^n.\end{equation}  
}

Since $\bar u_0$ is convex, then it also satisfies \eqref{convex}.
We  proceed  as in Proposition \ref{sconvex}. So for every $\eps>0$, there exists $R=R(\eps)>0$ such that $u_0(x)\leq \bar u_0(x)+\eps\leq \bar u(x,t)+\eps$ for all $|x|>R$ and moreover, arguing as  in the proof of \eqref{uno}, we get that   there exists $T>0$ sufficiently large such that \[\bar u(x,T)\geq   u_0(x)\qquad\text{ for all $|x|\leq R$. }\] 
Therefore, since $u_0(x)\leq \bar u(x, T)+\eps$, we conclude by comparison.

 \noindent {\bf Step 2:  for every $\delta>0$ there exist  $T(\delta)>0$ such that 
 \begin{equation}\label{loca} u(x,  t)\geq \bar u_0(x)-\delta\qquad \forall t\geq T(\delta).\end{equation}}
Since $\bar u_0$ is convex and is positively $1$-homogeneous, we get that  for all $\nu \in  \mathbb{S}^{n}$  there exists $p_\nu\in \R^n$,  such that $\bar u_0(x)\geq  p_\nu\cdot x$, with equality at every $x=\lambda\nu$ with $\lambda\geq 0$.  

We define the family of functions $u_0^\nu(x)=\inf (u_0(x), p_\nu\cdot x)$  and observe that by the assumption there holds that
\[\lim_{R\to +\infty} \sup_{\nu\in\mathbb{S}^n} \sup_{|x|>R} u_0^\nu(x)-p_\nu\cdot x=0. \] So, by Corollary \ref{stabilitafinale}, and arguing as in Remark \ref{uniform}, we get that 
  \[ \lim_{t\to +\infty}\sup_{\nu\in \mathbb{S}^{n}} d(E^\nu_t, H^\nu)=0\] where $H_\nu$ is the halfspace with normal $(-p_\nu, 1)$. This in particular implies that $\lim_{t\to +\infty} u^\nu(x,t)- p_\nu\cdot x=0$ uniformly in $\nu$, which in turns gives that $\liminf_{t\to +\infty} u(x,t)- p_\nu\cdot x\geq \lim_{t\to +\infty} u^\nu(x,t)- p_\nu\cdot x=0 $ uniformly in $\nu$, and so in particular $\liminf_{t\to +\infty} u(x,t)\geq \bar u_0(x)$. This permits to conclude.
  
  \noindent {\bf Step 3: conclusion.}
  
  Observe that by  Step 2, and comparison principle, for $\delta>0$ fixed,  there holds that  $u(x, t+T(\delta))\geq \bar u(x,t)-\delta$. 
  So, for every $\eps>0$ and $t\geq T(\delta)$, we get by Step 1 and the previous observation that 
  \[\bar u(x, t-T(\delta))-\bar u(x,t)-\delta \leq u(x,t)-\bar u(x,t)\leq \bar u(x, t+T(\eps))-\bar u(x,t)+\eps.\]
  Now we conclude by arbitrariness of $\eps,\delta$ and  by Remark \ref{convexcono}, letting $t\to +\infty$. 

\end{proof}

\section{Ancient  and homothetically shrinking solutions} \label{sectionsh} 
Finally we consider homothetically shrinking solutions in the graphical case, and we show that they are necessarily hyperplanes. 

 \begin{definition} \label{ancient} An  ancient solution  to \eqref{kflow}  is a solution to \eqref{kflow} defined for all $t\in (-\infty, 0)$.
  \end{definition} 
We recall the following Liouville theorem for ancient solutions of parabolic fractional equations with rough kernels, proved in  \cite[Theorem 3.1]{serra}. We state it in the setting we are going to apply it. 
\begin{theorem}\label{liouville}\cite[Theorem 3.1]{serra} Let $I$ be a translation invariant operator, elliptic with fractional order $1+s$, with $I(0)=0$ and $u\in C(\R^n\times (-\infty, 0])$ be a viscosity solution to 
$u_t-I(u)=0$  in $\R^n\times (-\infty, 0]$. Assume  there exists $C>0$ such that  for all $R\geq 1$ there holds
\[\sup_{|x|\leq R, -R^{1+s}\leq t\leq 0} |u(x,t)|\leq C R.\] Then there exists $a\in \R^n, b\in \R$ such that $u(x,t)=a\cdot x+b$.
\end{theorem} 
\begin{theorem}\label{nocontr}
 The only graphical Lipschitz solutions to \eqref{contr} are hyperplanes (with $c=0$). 
 
Moreover the only graphical uniformly Lipschitz ancient solutions to \eqref{kflow} are hyperplanes.
 \end{theorem} 
 
\begin{proof} The result is a consequence of Theorem \ref{liouville}. 

Let $E$ be a graphical Lipschitz solution to \eqref{contr}, that is let $u_1:\R^n\to \R$ be a Lipschitz continuous function such that $E=\{(x,z)\ | z\leq u_1(x)\}$ is a solution
 to \eqref{contr}. Then arguing as in  Proposition \ref{propchar} we may connstruct a solution to \eqref{levelset} in $(-\infty, 0)$ with $u_1(x,t)=u_1(x)$. Indeed  let
\[  E_t:= \left[-c(s+1)t+1\right]^{\frac{1}{s+1}} E=\lambda(t)E\qquad \text{ for }t<0.\]   It is easy to check, using the fact that $E$ is a solution to \eqref{contr}, that  $E_t$ is a solution to \eqref{kflow}. Therefore the function 
\begin{equation}\label{defu1} u_1(x,t):=\lambda(t)u_1\left( \frac{x}{\lambda (t)}\right)=\left[-c(s+1)t+1\right]^{\frac{1}{s+1}} u_1\left(x \left[-c(s+1)t+1\right]^{-\frac{1}{s+1}}\right) \end{equation} is a solution to \eqref{levelset} in $(-\infty, 0)$ and satisfies $u_1(x,0)=u_1(x)$. 
Since $u_1$ is Lipschitz continuous we get that
\[|u_1(x,t)|\leq \lambda(t)|u_1(0)|+\|Du_1\|_\infty |x|.\]
This implies that there exists $K>0$ depending on $c, s, \|Du_1\|_\infty, u_1(0)$, such that  for all $R>1$, \[\max_{|x|\leq R, t\in [-R^{s+1}, 0]}  |u_1(x,t)| \leq K R.\] 
Recalling the formula for $H_s$ \eqref{hsgraph1}, we get  that $u_1$ is a viscosity solution to 
\[u_t-I (u)=0, \qquad t\in (-\infty, 0), \] 
where $I$ is  a translation invariant operator, elliptic  with fractional order $s+1$, and $I(0)=0$. 
 Then, by Theorem \ref{liouville} 
we conclude that there exist $a\in \R^n$ and $b\in \R$ such that $u_1(x,t)= a\cdot x +b$ for all  $t\leq 0$ and $x\in\R^n$. 
This implies that $E$ is a hyperplane and $c=0$. 

Finally, if $E_t$ is a graphical uniformly Lipschitz  and ancient solution to \eqref{kflow}, then $u(x,t)$ is a continuous viscosity solution to $u_t-I (u)=0 $ for $t\in (-\infty, 0)$ and moreover, since $|Du(x,t)|\leq C$, arguing as in Proposition \ref{holder} we obtain that there exists a constant $K$ only depending on $C$ such that $|u(x,t)-u(x,t+h)|\leq K|h|^{\frac{1}{1+s}}$. So, again by Theorem \ref{liouville} we conclude that $u$ is affine and does not depend on $t$. 
\end{proof} 

 \begin{remark}\label{eternal}
 In the case of classical mean curvature flow, see for instance \cite{hof} and references therein, 
 there are translating, hence eternal, solutions which are smooth but not Lipschitz.
 We expect that such solutions, with polynomial growth depending on $s$, exist also for the graphical fractional mean curvature flow \eqref{levelset}.
 \end{remark}


\begin{bibdiv}
\begin{biblist}

\bib{av}{article}{
    AUTHOR = {Abatangelo, Nicola},
    author={ Valdinoci, Enrico},
     TITLE = {A notion of nonlocal curvature},
   JOURNAL = {Numer. Funct. Anal. Optim.},
    VOLUME = {35},
      YEAR = {2014},
    NUMBER = {7-9},
     PAGES = {793--815},
}


\bib{bfv}{article}{
author={ Barrios, Bego\~{n}a},
author={ Figalli, Alessio}, 
author ={Valdinoci, Enrico},
title={ Bootstrap regularity for integro-differential operators and its application to nonlocal minimal surfaces.},
journal={Ann. Sc. Norm. Super. Pisa Cl. Sci. (5)},
volume={ 13},
pages={609--639},
year={2014},
} 

\bib{MR2675483}{article}{
   author={Caffarelli, Luis},
   author={Roquejoffre, Jean-Michel},
   author={Savin, Ovidiu},
   title={Nonlocal minimal surfaces},
   journal={Comm. Pure Appl. Math.},
   volume={63},
   date={2010},
   number={9},
   pages={1111--1144},
   issn={0010-3640},
}

\bib{cs}{article}{
   author={Caffarelli, Luis},
   author={Souganidis, Panagiotis E.},
   title={Convergence of nonlocal threshold dynamics approximations to front
   propagation},
   journal={Arch. Ration. Mech. Anal.},
   volume={195},
   date={2010},
   number={1},
   pages={1--23},
   issn={0003-9527},
}


\bib{cdnv}{article}{
   author={Cesaroni, Annalisa},
   author={Dipierro, Serena},
   author={Novaga, Matteo},
   author={Valdinoci, Enrico},
   title={Fattening and nonfattening phenomena for planar nonlocal curvature flows},
  JOURNAL = {Math. Ann.},
    VOLUME = {375},
      YEAR = {2019},
    NUMBER = {1-2},
     PAGES = {687--736},
}


\bib{cn}{article}{
    AUTHOR = {Cesaroni, Annalisa},
    author={Novaga, Matteo},
     TITLE = {Symmetric self-shrinkers for the fractional mean curvature
              flow},
   JOURNAL = {J. Geom. Anal.},
    VOLUME = {30},
      YEAR = {2020},
    NUMBER = {4},
     PAGES = {3698--3715}, 
     }
	

\bib{ckn}{article}{
    AUTHOR = {Cesaroni, Annalisa},
    author={Kr\"oner, Heiko}, 
    author={Novaga, Matteo},
     TITLE = { Anisotropic mean curvature flow of Lipschitz graphs and convergence to self-similar solutions}, 
        JOURNAL = {ESAIM Control Optim. Calc. Var.},
    VOLUME = {27},
      YEAR = {2021},
         PAGES = {17 pp}, 
     }
     
\bib{cmp}{article}{
    AUTHOR = {Chambolle, Antonin},
    author={Morini, Massimiliano}, 
    author={Ponsiglione, Marcello},
     TITLE = {Nonlocal curvature flows},
   JOURNAL = {Arch. Ration. Mech. Anal.},
    VOLUME = {218},
      YEAR = {2015},
    NUMBER = {3},
     PAGES = {1263--1329},
}
	
 			

\bib{cnr}{article}{
   author={Chambolle, Antonin},
   author={Novaga, Matteo},
   author={Ruffini, Berardo},
   title={Some results on anisotropic fractional mean curvature flows},
   journal={Interfaces Free Bound.},
   volume={19},
   date={2017},
   number={3},
   pages={393--415},
   issn={1463-9963},
}

\bib{csv1}{article}{
   author={Cinti, Eleonora},
   author={Sinestrari, Carlo},
   author={Valdinoci, Enrico},
   title={Neckpinch singularities in fractional mean curvature flows},
   journal={Proc. Amer. Math. Soc.},
   volume={146},
   date={2018},
   number={6},
   pages={2637--2646},
   issn={0002-9939},
}	


\bib{csv2}{article}{
     author={Cinti, Eleonora},
   author={Sinestrari, Carlo},
   author={Valdinoci, Enrico},
        TITLE = {Convex sets evolving by volume-preserving fractional mean
              curvature flows},
   JOURNAL = {Anal. PDE},
    VOLUME = {13},
      YEAR = {2020},
    NUMBER = {7},
     PAGES = {2149--2171},
   }
		
\bib{clutt}{article}{ 
    AUTHOR = {Clutterbuck, Julie}, 
    author={Schn\"{u}rer, Oliver C.},
     TITLE = {Stability of mean convex cones under mean curvature flow},
   JOURNAL = {Math. Z.},
    VOLUME = {267},
      YEAR = {2011},
    NUMBER = {3-4},
     PAGES = {535--547},
}
 
 \bib{eh}{article}{
author = {Ecker, Klaus},
    author = {Huisken, Gerhard},
     TITLE = {Mean curvature evolution of entire graphs},
   JOURNAL = {Ann. of Math. (2)},
       VOLUME = {130},
      YEAR = {1989},
    NUMBER = {3},
     PAGES = {453--471},
}


 \bib{e}{book}{ 
    AUTHOR = {Ecker, Klaus},
     TITLE = {Regularity theory for mean curvature flow},
    SERIES = {Progress in Nonlinear Differential Equations and their
              Applications},
    VOLUME = {57},
 PUBLISHER = {Birkh\"{a}user Boston, Inc., Boston, MA},
      YEAR = {2004},
     PAGES = {xiv+165},
}
	


\bib{hof}{article}{   
 AUTHOR = {Hoffman, David}, 
AUTHOR = {Ilmanen, Tom},
 AUTHOR = {Mart\'{\i}n, Francisco}, 
 AUTHOR = {White, Brian},
     TITLE = {{G}raphical translators for mean curvature flow},
   JOURNAL = {Calc. Var. Partial Differential Equations},
     VOLUME = {58},
      YEAR = {2019},
    NUMBER = {4},
     PAGES = {Paper No. 158},
      ISSN = {0944-2669},
}
	

\bib{hu}{article}{
   author={Huisken, Gerhard},
   title={Flow by mean curvature of convex surfaces into spheres},
   journal={J. Differential Geom.},
   volume={20},
   date={1984},
   number={1},
   pages={237--266},
   issn={0022-040X},
}
%
%
\bib{i}{article}{
   author={Imbert, Cyril},
   title={Level set approach for fractional mean curvature flows},
   journal={Interfaces Free Bound.},
   volume={11},
   date={2009},
   number={1},
   pages={153--176},
   issn={1463-9963},
}


 \bib{lamanna}{article}{
 AUTHOR = {Julin, Vesa}, 
 author ={  La Manna, Domenico Angelo},
     TITLE = {Short time existence of the classical solution to the
              fractional mean curvature flow},
   JOURNAL = {Ann. Inst. H. Poincar\'{e} Anal. Non Lin\'{e}aire},
    VOLUME = {37},
      YEAR = {2020},
    NUMBER = {4},
     PAGES = {983--1016},
}

\bib{n}{article}{
    AUTHOR = {Nara, Mitsunori},
    author= {Taniguchi, Masaharu},
     TITLE = {The condition on the stability of stationary lines in a
              curvature flow in the whole plane},
   JOURNAL = {J. Differential Equations},
      VOLUME = {237},
      YEAR = {2007},
    NUMBER = {1},
     PAGES = {61--76},
}

 		

\bib{SAEZ}{article}{
   author = {{S{\'a}ez}, Mariel},
   author = {Valdinoci, Enrico},
    title = {On the evolution by fractional mean curvature},
  journal = {Comm. Anal. Geom.},
   volume={27},
   date = {2019},
    number={1},
    pages={211--249},
}

 \bib{serra}{article}{
  AUTHOR = {Serra, Joaquim},
     TITLE = {Regularity for fully nonlinear nonlocal parabolic equations
              with rough kernels},
   JOURNAL = {Calc. Var. Partial Differential Equations},
      VOLUME = {54},
      YEAR = {2015},
    NUMBER = {1},
     PAGES = {615--629},
}
\bib{silv}{article}{
    AUTHOR = {Schwab, Russell W.},
    author ={Silvestre, Luis},
     TITLE = {Regularity for parabolic integro-differential equations with
              very irregular kernels},
   JOURNAL = {Anal. PDE},
    VOLUME = {9},
      YEAR = {2016},
    NUMBER = {3},
     PAGES = {727--772},
}		
\bib{w}{article}{
    AUTHOR = {Wang, Lu},
     TITLE = {A {B}ernstein type theorem for self-similar shrinkers},
   JOURNAL = {Geom. Dedicata},
    VOLUME = {151},
      YEAR = {2011},
     PAGES = {297--303},
}
	
	 
\end{biblist}\end{bibdiv}
  \end{document}